\documentclass[a4paper,10pt]{amsart}

\usepackage[english]{babel}
\usepackage[utf8]{inputenc}
\usepackage{ulem} 

\usepackage{amssymb}
\usepackage{amsmath}
\usepackage{amsthm}
\usepackage{bbm}
\usepackage{mathtools}

\usepackage{enumerate}
\usepackage{tikz}
\usepackage{marginnote}

\newcommand{\bbB}{\mathbb{B}}
\newcommand{\bbC}{\mathbb{C}}

\newcommand{\bbN}{\mathbb{N}}

\newcommand{\bbP}{\mathbb{P}}
\newcommand{\bbQ}{\mathbb{Q}}
\newcommand{\bbR}{\mathbb{R}}
\newcommand{\bbS}{\mathbb{S}}
\newcommand{\bbT}{\mathbb{T}}

\newcommand{\calB}{\mathcal{B}}

\newcommand{\calD}{\mathcal{D}}

\newcommand{\calL}{\mathcal{L}}
\newcommand{\calM}{\mathcal{M}}

\newcommand{\calR}{\mathcal{R}}

\newcommand{\calT}{\mathcal{T}}

\newcommand{\N}{\bbN}
\newcommand{\R}{\bbR}
\newcommand{\C}{\bbC}
\newcommand{\Q}{\bbQ}

\DeclarePairedDelimiter{\norm}{\lVert}{\rVert}
\DeclarePairedDelimiter{\abs}{\lvert}{\rvert}
\DeclarePairedDelimiter{\dual}{\langle}{\rangle}
\DeclarePairedDelimiter{\set}{\lbrace}{\rbrace}

\DeclareMathOperator{\id}{id}
\DeclareMathOperator{\one}{\mathbbm{1}}

\DeclareMathOperator{\dom}{dom}

\newcommand{\ue}{\mathrm{e}}
\newcommand{\ui}{\mathrm{i}}
\newcommand{\ud}{\mathrm{d}}

\DeclareMathOperator{\diag}{diag}

\theoremstyle{definition}
\newtheorem{definition}{Definition}[section]
\newtheorem{remark}[definition]{Remark}

\theoremstyle{plain}
\newtheorem{proposition}[definition]{Proposition}
\newtheorem{lemma}[definition]{Lemma}
\newtheorem{theorem}[definition]{Theorem}
\newtheorem{corollary}[definition]{Corollary}

\numberwithin{equation}{section}

\begin{document}

\normalem

\title[Network Transport]{On the asymptotic behaviour of semigroups for flows in infinite networks}
%

\author{Alexander Dobrick}
\address{Alexander Dobrick, Arbeitsbereich Analysis, Christian-Albrechts-Universit\"at zu Kiel, Ludewig-Meyn-Str.\ 4, 24098 Kiel, Germany}
\email{dobrick@math.uni-kiel.de}

\date{\today}
\begin{abstract}
We study transport processes on infinite networks. The solution of these processes can be modeled by an operator semigroup on a suitable Banach space. Classically, such semigroups are strongly continuous and therefore their asymptotic behaviour is quite well understood. However, recently new examples of transport processes emerged where the corresponding semigroup is not strongly continuous. Due to this lack of strong continuity, there are currently only few results on the long-term behaviour of these semigroups. In this paper, we discuss the asymptotic behaviour for a certain class of these transport processes. In particular, it is proved that the solution semigroups behave asymptotically periodic with respect to the operator norm as a consequence of a more general result on the long-term behaviour by positive semigroups containing a multiplication operator. Furthermore, we revisit known results on the asymptotic behaviour of transport processes on infinite networks and prove the asymptotic periodicity of their extensions to the space of bounded measures. 
\end{abstract}

\maketitle

\section{Introduction}
Consider a transport process on an infinite network, modeled by an infinite, directed graph $G = (V, E)$, which is assumed to be simple, locally finite and non-degenerate. Moreover, one can consider $G$ as a metric graph by identifying each edge with the unit interval $[0, 1]$ and parametrizing it contrarily to its direction (see Sections~\ref{section:preliminaries}~and~\ref{section:applications-to-transport-equations-in-infinite-networks} for definitions).

The distribution of mass transported along one edge $e_j$, $j \in J \subseteq \N$, at some time $t \geq 0$ is described by a function $u_j(t, x)$ for $x \in [0, 1]$ and the material is transported along $e_j$ with a constant velocity $c_j > 0$ satisfying
\begin{align*}
0 < c_{\min} \leq c_j \leq c_{\max} < \infty.
\end{align*}
Define $\bbB^C := C^{-1} \bbB C$, where $\bbB$ denotes the weighted (transposed) adjacency matrix of the graph $G$ and $C := \diag(c_j)$ denotes the (diagonal) velocity matrix. In addition, suppose that the functions $u_j$ satisfy the generalized Kirchhoff law 
\begin{align*}
\sum_{j \in J} \phi^-_{ij} c_j u_j(1, t) = \sum_{j \in J} \phi^+_{ij} c_j u_j(0, t)
\end{align*}
for all $i \in I$ and $t > 0$. Then the transport process can be modeled by the initial value problem
\begin{align} \label{eq:intro-network-transport-equation}
\left\lbrace 
\begin{aligned}
\frac{\partial}{\partial t} u_j(t, x) &= c_j \frac{\partial}{\partial x} u_j(t, x), &&\qquad x \in (0, 1), \, t \geq 0, \\
u_j(0, x) &= f_j(x), &&\qquad x \in (0, 1),\\
u_j(1, t) &= \sum_{k \in J} \bbB^C_{jk} u_k(0, t), &&\qquad t \geq 0,\\
\end{aligned}
\right.
\end{align}
where $f_j$, $j \in J$, are the initial distributions of mass along the edges of $G$.

The investigation of systems of the form \eqref{eq:intro-network-transport-equation} on metric graphs by methods from the theory of strongly continuous semigroups has a history. The first results in this direction are due to Kramar and Sikolya (see \cite{Kramar2005}). This paper was followed by a series of papers \cite{Engel2008, Dorn2008, Dorn2009, Dorn2010, Bayazit2013} from several different authors studying various transport processes. In all these papers, the state space is $L^1([0, 1]; \ell^1)$, where the solution semigroups turn out to be strongly continuous, and, the asymptotic behaviour of the solution semigroup is a major point of interest.

Motivated by results from \cite{Below2005, Below2009}, Budde and Kramar Fijav\v{z} in \cite{Budde2019} discuss transport processes on the state space $L^\infty([0, 1]; \ell^1)$. In this setting the solution semigroup is bi-continuous with respect to the weak$^\ast$-topology on $L^\infty([0, 1]; \ell^1)$ (see \cite{Kuehnemund2003} for a definition) but not strongly continuous. However, the asymptotic behaviour of the solutions in the bi-continuous case is not discussed.

In the present paper, we shall discuss the asymptotic behaviour in the bi-continuous case (see Theorem~\ref{theorem:asymptotic-periodicity-mult-op-semigroup-lp}) by combining spectral theoretic observations, the concept of the semigroup at infinity (see \cite{Dobrick2020}) and classical Perron-Frobenius theory. In particular, our approach does not make use of any regularity assumptions on the semigroup. Moreover, it allows us to revisit a result on the asymptotic behaviour from \cite{Dorn2008} (see Proposition~\ref{proposition:l1-asymptotics-simple-case}) and improve slightly upon the statement as well as its proof. 

Finally, inspired by \cite{Matrai2007}, we extend the semigroup on $L^1([0, 1]; \ell^1)$ from \cite{Dorn2008} to the space of $\ell^1$-valued measures of bounded variation and show that their extension has still some regularity and converges with respect to operator norm exponentially fast to a periodic semigroup.

\subsection*{Structure of the article}
In Section~\ref{section:preliminaries} we first recall some basic facts from graph theory and on metric graphs. Moreover, we discuss vector-valued multiplication operators on $L^p(\Omega; E)$ for a given Banach space $E$ and a given $\sigma$-finite measure space $(\Omega, \Sigma, \mu)$, and prove some spectral theoretic results. In Section~\ref{section:abstract-asymptotics-result} we utilize our spectral observations from Section~\ref{section:preliminaries} to prove an abstract result on the long-term behaviour for positive semigroups  containing a multiplication operator. In Section~\ref{section:applications-to-transport-equations-in-infinite-networks} we apply that result to several different transport semigroups on infinite networks and we prove that, under certain conditions, these transport semigroups are asymptotically periodic with respect to the operator norm. Finally, in Section~\ref{section:applications-to-transport-equations-in-infinite-networks-measures} we extend the solution semigroup of \eqref{eq:intro-network-transport-equation} from the state space of $\ell^1$-valued $L^1$-functions to the space of $\ell^1$-valued measures of bounded variation and prove that this extension is asymptotically periodic with respect to the operator norm. 

\subsection*{Notation and terminology}
Let $E,F$ be Banach spaces. We endow the space $\calL(E;F)$ of bounded linear operators from $E$ to $F$ with the operator norm topology throughout; moreover, we use the abbreviation $\calL(E) := \calL(E;E)$. Further a \textit{semigroup} is a map $T: [0, \infty) \to \calL(E)$ with $T(0) = \id_E$ satisfying the \textit{semigroup law} $T(s + t) = T(s) T(t)$ for all $s, t \geq 0$. In this case, we call $(T(t))_{t \geq 0}$ an \textit{operator semigroup}. Note that we do not make any assumption on the regularity of $(T(t))_{t \geq 0}$. An operator $T \in \calL(E)$ is called \textit{quasi-compact} if it has less than distance one in the operator norm to the space of compact operators on $E$. The dual Banach space of $E$ will be denoted by $E'$. If $(\Omega, \Sigma, \mu)$ is a measure space, then $L^p(\Omega; E)$, $1 \leq p \leq \infty$, denotes the space of $L^p$-Bochner-integrable functions with values in $E$. For the definition of the space $M(\Omega; E)$ we refer to Appendix~\ref{subsection:vector-valued-measures}. For $f \in L^p(\Omega)$ and $x \in E$ we define $f \otimes x \in L^p(\Omega; E)$ by $(f \otimes x)(s) \coloneqq f(s) x$ for almost all $s \in \Omega$. 

Throughout the paper, we will make free use of the theory of real and complex Banach lattices (cf.\ \cite{Meyer-Nieberg1991, Schaefer1974, Zaanen1997}). Here we only recall that if $E$ is a Banach lattice, then an operator $T: E \to E$ is called \textit{positive} if $Tf \geq 0$ for all $0 \leq f \in E$ and \textit{strictly positive} if $Tf > 0$, whenever $0 < f \in E$. Moreover, $T$ is said to be \textit{irreducible} if for each $0 < f \in E$ and $0 < f' \in E'$ there exists $n \in \N$ such that $\dual{T^n f, f'} > 0$.  

If the underlying scalar field is complex, the spectrum of a linear operator \linebreak $A: E \supseteq D(A) \to E$ will be denoted by $\sigma(A)$; for $\lambda \in \bbC \setminus \sigma(A)$, the resolvent of $A$ at $\lambda$ is denoted by $R(\lambda, A) := (\lambda - A)^{-1}$. Further, the point spectrum of $A$ will be denoted by $\sigma_p(A)$. If $T$ is a bounded operator on $E$, then $r(T)$ denotes the spectral radius of $T$. Finally, the set $\set{\lambda \in \sigma(T) : \abs{\lambda} = r(T)}$  is called the \textit{peripheral spectrum} of the operator $T$. 

\subsection*{Acknowledgements}

The author thanks Abdelaziz Rhandi for his very kind hospitality at the university of Salerno in Spring 2020 at the begin of the extraordinary Covid-19 situation in Italy. At this time major parts of this paper have been written. Moreover, I am indebted to my dear colleagues Marco Peruzzetto for several helpful discussions and Jochen Glück for helpful feedback at an early stage of the paper and bringing reference \cite{Matrai2007} to my attention. Finally, I thank Markus Haase for his many fruitful suggestions.

\section{Preliminaries} \label{section:preliminaries}

\subsection{Basic Graph Theory}

A \textit{directed graph} $G$ is a pair $(V, E)$, where $V = \set{v_i : i \in I}$ is a \textit{set of vertices} and $E = \set{e_j : j \in J} \subseteq V \times V$ a \textit{set of directed edges} for index sets $I, J \subseteq \N$. If $e = (v_i, v_j) \in E$ is some directed edge, then $v_i$ is called the \textit{tail} and $v_j$ is called the \textit{head} of $e$. In particular, $e$ is considered to be directed from $v_i$ to $v_j$ and we call $e$ and \textit{outgoing edge} with respect to $v_i$ and analogously an \textit{incoming edge} with respect to $v_j$.

A very natural way to describe the structure of graphs are \textit{incidence matrices}. The \textit{incoming incidence matrix} $\Phi^+ =(\phi_{i j}^+)_{i \in I, j \in J}$ is defined by
\begin{align*}
\phi_{i j}^+ := 
\begin{cases}
1, &\quad  \text{if } v_i \text{ is the head of } e_j, \\
0, &\quad \text{else},
\end{cases}
\end{align*}
and the \textit{outgoing incidence matrix} $\Phi^- =(\phi_{i j}^-)_{i \in I, j \in J}$ by 
\begin{align*}
\phi_{i j}^- := 
\begin{cases}
1, &\quad  \text{if } v_i \text{ is the tail of } e_j, \\
0, &\quad \text{else}.
\end{cases}
\end{align*}
As the graph $G$ will be assumed to be \textit{weighted} in what follows, we consider also the \textit{weighted outgoing incidence matrix} $\Phi_w^- =(\phi_{w, i j}^-)_{i \in I, j \in J}$ which is defined by
\begin{align*}
\phi_{w, i j}^- := 
\begin{cases}
w_{ij}, &\quad  \text{if } v_i \text{ is the tail of } e_j, \\
0, &\quad \text{else}.
\end{cases}
\end{align*}
Here $w_{ij} \in [0, 1]$ are weights for all $i \in I$ and $j \in J$ such that 
\begin{align} \label{eq:conservative-weights}
\sum_{j \in J} w_{ij} = 1.
\end{align}
This condition ensures that no material of the network flow is absorbed by the network.

Another way to describe the structure of a given graph is the \textit{adjecency matrix}. For our approach \textit{the weighted (transposed) adjacency matrix} $\bbB = (\bbB_{i, j})_{i \in I, j \in J}$ graph defined by $\bbB := (\Phi_w^-)^T \Phi^+$ will play the most prominent role. Its entries are given by
\begin{align*}
\bbB_{ij} = 
\begin{cases}
w_{ki}, &\quad  \text{if } v_k \text{ is the head of } e_j \text{ and the tail of } e_i, \\
0, &\quad \text{else}.
\end{cases}
\end{align*}
By positivity of the weights and \eqref{eq:conservative-weights}, $\bbB$ is a column-stochastic matrix and thus corresponds to a stochastic operator on $\ell^1$ with $r(\bbB) = \norm{\bbB} = 1$. \smallskip

A \textit{directed path} $p$ from the vertex $v \in V$ to the vertex $w \in V$ is a tuple $(e_1, \dots, e_\ell)$ of edges $e_1, \dots, e_\ell \in E$ such that the tail of $e_1$ is $v$, the head of $e_\ell$ is $w$ and the head of $e_i$ is the tail of $e_{i + 1}$ for all $i = 1, \dots, \ell - 1$. In this case, $v$ is called the \textit{starting point} and $w$ the \textit{endpoint} of $p$ and $\ell$ is called the \textit{length} of $p$.
A directed graph $G = (V ,E)$ is called \textit{strongly connected} if for any $v, w \in V$ there is a directed path of finite length between them. 

A finite vertex set $W \subseteq V$ is called an \textit{attractor} if there exist $L \in \N$ and $\delta > 0$ such that for each vertex $v \in V$ there is a path $p$ with head $v$ of length at most $L$ that has an endpoint if $W$ and such that the sum of the weights of all paths with this property is greater or equal to $\delta$. Here the \textit{weight of a path} $p = (e_1, \dots, e_\ell)$ is defined by the product of all weights belonging to its edges $e_i$. 

Many graph-theoretic properties of $G = (V, E)$ translate into properties of the operator $\bbB$ and vice versa. The following proposition collects some of these correspondences (cf.\ \cite[Proposition~4.8, Proposition~4.9]{Dorn2008}) that are relevant to this paper.

\begin{proposition} \label{proposition:graph-properties-vs-bbB-properties}
Let $G$ be a locally finite graph. Then the following assertions hold:
\begin{enumerate}[\upshape (i)]
\item $G$ is strongly connected if and only if the adjacency matrix $\bbB$ is irreducible. 
\item $G$ has an attractor if and only if the adjacency matrix $\bbB$ is quasi-compact. 
\end{enumerate}
\end{proposition}

\subsection{Vector-valued multiplication operators}
Let $T$ be a bounded operator on a Banach space $E$, $(\Omega, \Sigma, \mu)$ a $\sigma$-finite measure space and $1 \leq p \leq \infty$. Consider the operator-valued multiplication operator
\begin{align*}
\calM_T: L^p(\Omega; E) \to L^p(\Omega; E), \quad (\calM_T f)(s) := Tf(s).
\end{align*}
It is easy to see that this operator is well-defined and it is easy to see that $\calM_T$ is a bounded operator. Moreover, the map
\begin{align*}
\calM: \calL(E) \to \calL(L^p(\Omega; E)), \quad T \mapsto M_T
\end{align*}
is an isometric, unital algebra homomorphism (cf.\ \cite[Proposition~2.2.14]{Thomaschewski2003} for a more general result). In particular, projections are mapped to projections by $\calM$. Finally, in the case that $E$ is a Banach lattice, $\calM_T$ is positive if and only if $T$ is positive and $\calM$ maps strictly positive operators to strictly positive operators. 

The next lemma lists some rather obvious properties.

\begin{lemma} \label{lemma:mult-op-invariant-subspaces}
Let $T$ be a bounded operator on a Banach space $E$. Then the following assertions hold:
\begin{enumerate}[\upshape (i)]
\item If $F$ is a closed, $T$-invariant subspace of $E$, then $L^p(\Omega; F)$ is a closed $\calM_T$-invariant subspace of $L^p(\Omega; E)$ and $\calM_{T \rvert_F} = \calM_{T}\rvert_{L^p(\Omega; F)}$.
\item If $E$ decomposes as $E = E_1 \oplus E_2$ into closed $T$-invariant subspaces $E_1$ and $E_2$, then $L^p(\Omega; E)$ decomposes as $L^p(\Omega; E) = L^p(\Omega; E_1) \oplus L^p(\Omega; E_2)$ into closed $\calM_T$-invariant subspaces and $\calM_T = \calM_{T\rvert_{E_1}} \oplus \calM_{T\rvert_{E_2}}$. 
\end{enumerate}
\end{lemma}

Now we turn to spectral theoretic results. A simple calculation shows that $\sigma(T) = \sigma(\calM_T)$ with
\begin{align*}
R(\lambda, \calM_T) = \calM_{R(\lambda, T)} \qquad (\lambda \in \rho(T)). 
\end{align*}

Recall that in case the spectrum of $T$ contains isolated subsets, one can define spectral projection of $T$ by virtue of the Dunford functional calculus: Namely, for any isolated subset $\sigma$ of $\sigma(T)$, the spectral projection $P_\sigma$ of $T$ associated to $\sigma$ is
\begin{align*}
P_\sigma := \frac{1}{2 \pi \ui} \int_\gamma R(z, T) \, \ud z,
\end{align*} 
where $\gamma$ is a contour in $\rho(T)$ such that the winding number around each element of $\sigma$ is equal to $1$ and around each element of $\sigma(T) \setminus \sigma$ is equal to $0$. Note that such a contour always exists and that Dunford's integral does not depend on the choice of the contour by Cauchy's theorem. 

In the situation above, as the spectra of $T$ and $\calM_T$ coincide, it is natural to ask, whether the spectral projections of $T$ and $\calM_T$ associated to $\sigma$ are related to each other. This is the content of the following result.

\begin{proposition} \label{proposition:spectral-projection-multiplication-operator}
Let $(\Omega, \Sigma, \mu)$ be a $\sigma$-finite measure space, $E$ a Banach space and $X := L^p(\Omega; E)$, $1 \leq p \leq \infty$. Let $T$ be a bounded operator on $E$, let $\sigma$ be an isolated subset of $\sigma(T)$ and $P$ the spectral projection of $T$ associated to $\sigma$. Then the following assertions hold:
\begin{enumerate}[\upshape (i)]
\item $X$ decomposes as $X = \calM_P X \oplus \ker \calM_P$ into the $\calM_T$-invariant subspaces $\calM_P X = L^p(\Omega; PE)$ and $\ker \calM_P = L^p(\Omega; \ker P)$.
\item The spectral projection of $\calM_T$ associated to $\sigma$ coincides with the operator-valued multiplication operator $\calM_P$.
\end{enumerate}
\end{proposition}

\begin{proof}[Proof~of~Proposition~\ref{proposition:spectral-projection-multiplication-operator}]
Although the following proof is almost trivial, we present it here for the convenience of the reader:

(i) $E$ can be decomposed as $E = PE \oplus \ker P$. Moreover, $PE$ and $\ker P$ are $T$-invariant. Therefore, Lemma~\ref{lemma:mult-op-invariant-subspaces}(ii) yields the decomposition
\begin{align*}
X = L^p(\Omega; PE) \oplus L^p(\Omega; \ker P)
\end{align*}
into $\calM_T$-invariant subspaces. Finally, we have $f \in \calM_P X$ ($f \in \ker \calM_P$) if and only if $f(\omega) \in P E$ ($f(\omega) \in \ker P$) for a.e. $\omega \in \Omega$. Therefore, $\calM_P X = L^p(\Omega; PE)$ and $\ker \calM_P = L^p(\Omega; \ker P)$.

(ii) By (i), one has the decomposition $X = \calM_P X \oplus \ker \calM_P$ into $\calM_T$-invariant subspaces. Furthermore, 
\begin{align*}
\sigma(\calM_T\rvert_{\calM_P X}) &= \sigma(\calM_T\rvert_{L^p(\Omega; PE)}) = \sigma(\calM_{T\rvert_{PE}}) = \sigma(T\rvert_{PE}) = \sigma, \\
\sigma(\calM_T\rvert_{\ker \calM_P}) &= \sigma(\calM_T\rvert_{L^p(\Omega; \ker P)}) = \sigma(\calM_{T\rvert_{\ker P}}) = \sigma(T\rvert_{\ker P}) = \sigma(T) \setminus \sigma,
\end{align*}
by Lemma~\ref{lemma:mult-op-invariant-subspaces}(i). Thus, the claim follows from the uniqueness of the spectral projections on $\sigma$ and $\sigma(T) \setminus \sigma$ of $\calM_T$, respectively.
\end{proof}

The next result is concerned with isolated points of the spectrum of operator-valued multiplication operators and their spectral properties. The first assertion in the following theorem is a special case of \cite[Theorem~6]{Heymann2014}.

\begin{proposition} \label{prop:mult-op-spectral-properties}
Let $(\Omega, \Sigma, \mu)$ be a $\sigma$-finite measure space, let $E$ be a Banach space, $X := L^p(\Omega; E)$, $1 \leq p \leq \infty$ and let $T \in \calL(E)$. Then the following assertions hold:
\begin{enumerate}[\upshape (i)]
\item $\sigma_p(T) = \sigma_p(\calM_T)$.
\item If $\lambda$ is an eigenvalue of $T$, then its geometric (algebraic) multiplicity with respect to $T$ is less or equal to its geometric (algebraic) multiplicity with respect to $\calM_T$. 
\item $\lambda$ is a pole of order $k$ of the resolvent of $T$ if and only if $\lambda$ is a pole of order $k$ of the resolvent of $\calM_T$. 
\end{enumerate}
\end{proposition}

\begin{proof}
(i) Let $\lambda \in \sigma_p(T)$ and $x \in E$ an associated eigenvalue. Choose any non-zero $f \in L^p(\Omega)$. Then it is easy to see that $(\lambda - \calM_T)(f \otimes x) = 0$. Hence, $\sigma_p(T) \subseteq \sigma_p(\calM_T)$. 

On the other hand, suppose that $\lambda \notin \sigma_p(T)$ and that there is $f \in L^p(\Omega; E)$ such that $(\lambda - \calM_T) f = 0$. Then $(\lambda - T)f(\omega) = 0$ and therefore $f(\omega) = 0$ for a.e.\ $\omega \in \Omega$. Hence, $f = 0$, i.e. $\lambda \notin \sigma_p(\calM_T)$. 

(ii) Let $x \in E$ such that $(\lambda - T)^n x = 0$ for some $n \in \N$. Choose any non-zero $f \in L^p(\Omega)$. Then 
\begin{align*}
(\lambda - \calM_T)^n (f \otimes x) = \calM_{(\lambda - T)^n} (f \otimes x) = f \otimes (\lambda - T)^n x = 0,
\end{align*}
which shows the claim.

(iii) Suppose that $\lambda$ is an isolated point of the spectrum of $T$, or equivalently, of $\calM_T$. Consider the Laurent expansion the respective resolvents about $\lambda$ given by
\begin{align*}
R(\mu, T) = \sum_{n = -\infty}^\infty U_n (\mu - \lambda)^n, \qquad R(\mu, \calM_T) = \sum_{n = -\infty}^\infty V_n (\mu - \lambda)^n
\end{align*}
for coefficients $U_n \in \calL(E)$ and $V_n \in \calL(L^p(\Omega; E))$. We show that $V_n = \calM_{U_n}$ for all $n \leq -1$. The case $n = -1$ follows directly from Proposition~\ref{proposition:spectral-projection-multiplication-operator}(ii) as the residua in $\lambda$ coincide with the respective spectral projections. For the general case, observe
\begin{align*}
V_{-n} = (\calM_T - \lambda)^{n - 1} V_{-1} = \calM_{(T - \lambda)^{n - 1}} \calM_{U_{-1}} = \calM_{(T - \lambda)^{n - 1} U_{-1}} = \calM_{U_{-n}}
\end{align*}
by \cite[Formula IV.1.13]{Engel2000}. As $\calM$ is injective, one has $V_{-n} = 0$ if and only if $U_{-n} = 0$. Hence, the claim follows.
\end{proof}

\begin{remark} \label{remark:multiplication-operators-on-measures}
Let $T$ be a bounded operator on a Banach space $E$ and $(\Omega, \Sigma)$ a measurable space. Then it is not hard to show that the operator-valued multiplication operator
\begin{align*}
\calM_T : M(\Omega; E) \to M(\Omega; E), \quad (\calM_T \mu)(A) := T \mu(A)
\end{align*}
is a well-defined bounded operator. However, in contrast to the $L^p$-case, the map
\begin{align*}
\calM : \calL(E) \to \calL(M(\Omega; E)), \quad T \mapsto \calM_T
\end{align*} 
is just an injective algebra homomorphism and not necessarily isometric. On the other hand, this operator shares all other properties of the operator-valued multiplication operator mentioned in the beginning of this section. For this reason all proofs of the propositions that followed can be carried out for multiplication operators on $M(\Omega; E)$; this yields that all of these propositions hold for multiplication operators on $M(\Omega; E)$ as well.
\end{remark}

\section{An abstract result on semigroup asymptotics} \label{section:abstract-asymptotics-result}

The asymptotic behaviour of $C_0$-semigroups is a topic widely covered in the existing literature and much less is known when the semigroups are not strongly continuous. Recently, however, recently several authors investigated the asymptotic behaviour of semigroups with weaker regularity properties or even no at all (cf.\ \cite{GerlachLB}, \cite{Gerlach2017}, \cite{Glueck2019}, \cite{Dobrick2020}). In \cite{Glueck2019} Glück and Haase introduced the concept of the semigroup at infinity to investigate the asymptotic behaviour of semigroup representations with respect to the strong operator topology. Following this idea, Glück and the author of the present paper adapted this concept in \cite{Dobrick2020} to study convergence of semigroup representations with respect to the operator topology. In the present paper we use the semigroup at infinity to prove an abstract convergence result, which will be then applied to network flows in Section~\ref{section:applications-to-transport-equations-in-infinite-networks}. \smallskip

Let $T$ be a bounded operator on a Banach space $E$. Then $\calT:= \{T^n : n \in \bbN_0 \}$ is a semigroup of operators, and we call the set
\begin{align*}
\calT_\infty := \bigcap_{m \in \N} \overline{\set{T^n : n \geq m}}
\end{align*}
the \textit{semigroup at infinity} associated with the semigroup $\calT$. If $\calT_\infty$ is non-empty and compact, then one can apply the Jacobs-de Leeuw-Glicksberg decomposition (see e.g.\ \cite[Chapter~16]{Eisner2015}) to the semigroup $\calT_\infty$ to obtain the \textit{projection at infinity} $P_\infty \in \calT_\infty$ (see the discussion before \cite[Theorem~4.3]{Dobrick2020}). Using the semigroup at infinity we can easily prove the following lemma.

\begin{lemma} \label{lemma:positive-spectral-projection}
Let $T$ be a power-bounded, positive operator on a Banach lattice $E$ such that all spectral values of $T$ on the unit circle are poles of the resolvent. Then the spectral projection of $T$ associated with $\sigma(T) \cap \bbT$ is positive.
\end{lemma}

\begin{proof}
Let $P$ be the spectral projection of $T$ associated with $\sigma(T) \cap \bbT$. By \cite[Proposition~4.12]{Dobrick2020}, the semigroup at infinity $\calT_\infty$ associated to the semigroup $\calT:= \{T^n : n \in \bbN_0 \}$ is non-empty and compact. Since $\calT$ consists only of positive operators and since the projection at infinity $P_\infty$ is contained in the closure of $\calT$, it follows that $P_\infty$ is positive. Now the claim follows as $P_\infty = P$ by \cite[Proposition~4.12]{Dobrick2020}.
\end{proof}

The following lemma is basically the spectral theoretic core of the arguments in the proof of \cite[Proposition~4.3]{Dorn2008}. However, assertion~(iii) below is missing in \cite{Dorn2008}. In the following section, we will be mainly interested in the case that $T$ is an infinite column-stochastic matrix with additional properties.

\begin{lemma} \label{lemma:value-space-splitting}
Let $T$ be a power-bounded, positive and irreducible operator on a Banach lattice $E$ such that $r(T) = 1$ is a pole of the resolvent of $T$. Then the following assertions hold:
\begin{enumerate}[\upshape (i)]
\item The peripheral spectrum of $T$ consists of the $k$th roots of unity for some $k \in \N$. Moreover, these roots of unity are all first order poles of the resolvent. 
\item The spectral projection $P$ of $T$ associated to the peripheral spectrum yields a decomposition $E = PE \oplus \ker P$ such that $T^k\rvert_{PE} = \id_{PE}$ and $r(T\rvert_{\ker P}) < 1$, where $k \in \N$ is the index of imprimitivity of $T$.
\item $P$ is strictly positive. In particular, $PE$ is a sublattice of $E$.
\end{enumerate}
\end{lemma}

\begin{proof}
Since $T$ is positive and irreducible and $r(T)$ is a pole of the resolvent of $T$, it follows from \cite[Theorem~V.5.4]{Schaefer1974} that the peripheral spectrum of $T$ is cyclic and consists of first order poles of the resolvent, only. In particular, the peripheral spectrum of $T$ is a finite group consisting of the $k$th roots of unity for some $k \in \N$. This shows (i).

In the following we denote the primitive $k$th root of unity by $\zeta$. By spectral decomposition it is clear that $T$ acts on $PE$ as the multiplication with the tuple $(1, \zeta, \dots, \zeta^{k - 1})$ and $r(T\rvert_{\ker P}) < 1$. Hence, $T^k$ acts as the identity on $PE$ which shows (ii).

To show (iii), observe first that $P$ is positive by Lemma~\ref{lemma:positive-spectral-projection}. To see that $P$ is even strictly positive, we consider the absolute kernel 
\begin{align*}
\calR := \set{x \in E: P \abs x = 0}.
\end{align*}
Clearly, $\calR$ is an ideal in $E$ and that is enough to show that $\calR = \set 0$. For each $x \in \calR$ one has
\begin{align*}
0 \leq P \abs{T x} \leq P T \abs x = T P \abs x = 0
\end{align*}
and therefore $P \abs{T x} = 0$. Hence, $\calR$ is a $T$-invariant ideal and thus trivial by the irreducibility of $T$. Since $P \neq 0$, one has $\calR = \set 0$, i.e., $P$ is strictly positive. Finally \cite[Proposition~III.11.5]{Schaefer1974} shows that $PE$ is a sublattice of $E$.
\end{proof}

\begin{definition} \label{def:asymptotically-periodic}
We call a bounded semigroup $(T(t))_{t \geq 0}$ on a Banach space $E$ \emph{asymptotically exponentially periodic} on $E$ if there exists a projection $P$ that commutes with the semigroup $(T(t))_{t \geq 0}$ such that
\begin{enumerate}[\upshape (i)]
\item $(T(t)\rvert_{PE})_{t \geq 0}$ can be extended to a periodic group on $PE$. 
\item $T(t)\rvert_{\ker P}$ is uniformly exponentially stable as $t \to \infty$.
\end{enumerate}
\end{definition} 

The following theorem is a generalized version of \cite[Proposition~4.3]{Dorn2008} by Dorn. In contrast to Dorn's result, it does not require any regularity assumptions on the semigroup. Hence, it is applicable to semigroups with no or weaker regularity properties, e.g., to bi-continuous semigroups.

\begin{theorem} \label{theorem:asymptotic-periodicity-mult-op-semigroup-lp}
Let $\Omega = (\Omega, \Sigma, \mu)$ be a measure space, $E$ a Banach lattice and $(T(t))_{t \geq 0}$ a bounded, positive semigroup on $X := L^p(\Omega; E)$, $1 \leq p \leq \infty$. Suppose that there is $t_0 \in [0, \infty)$ such that $T(t_0) = \calM_B$ for some irreducible operator $B$ on $E$ with $r(B) = 1$. If $r(B)$ is a pole of the resolvent of $B$, then there is a strictly positive projection $P$ commuting with $(T(t))_{t \geq 0}$ and with the following properties:
\begin{enumerate}[\upshape (i)]
\item $(T(t)\rvert_{PX})_{t \geq 0}$ can be extended to a positive, periodic group on the Banach lattice $PX$.
\item $(T(t)\rvert_{\ker P})_{t \geq 0}$ is uniformly exponentially stable, i.e., there exist $M \geq 1$ and $\omega > 0$ such that $\norm{T(t) - T(t) P}  \leq M \ue^{-\omega t}$. 
\end{enumerate}
\end{theorem}

\begin{proof}
As $(T(t))_{t \geq 0}$ is a bounded, positive semigroup, $B$ is power-bounded and positive. Hence, $B$ satisfies the assumptions of Lemma~\ref{lemma:value-space-splitting} which yields the decomposition $E = PE \oplus \ker P$, where $P$ denotes the spectral projection of $B$ associated to its peripheral spectrum. Moreover, $B^k$ acts as the identity on $PE$, where $k \in \N$ is the index of imprimitivity of $B$, and $r(B\rvert_{\ker P}) < 1$. Proposition~\ref{proposition:spectral-projection-multiplication-operator} yields the decomposition
\begin{align*}
X = \calM_P X \oplus \ker \calM_P = L^p(\Omega; PE) \oplus L^p(\Omega; \ker P) =: X_1 \oplus X_2,
\end{align*}
which coincides with the spectral decomposition corresponding to $\calM_B$. 

Now observe that $T(t_0)^k = \calM_{B^k}$ and therefore
\begin{align*}
T(t_0)^k\rvert_{X_1} = \calM_{B^k}\rvert_{X_1} = \calM_{B^k\rvert_{P E}} = \calM_{\id_{P E}} = \id_{X_1}
\end{align*}
by Lemma~\ref{lemma:mult-op-invariant-subspaces}(i). Furthermore,
\begin{align*}
r(T(t_0)\rvert_{X_2}) = r(\calM_{B\rvert_{\ker P}}) = r(B\rvert_{\ker P}) < 1.
\end{align*}
So the semigroup $(T(t_0)^n\rvert_{X_1})_{n \in \N_0}$ is periodic. Furthermore, there exists $q \in (\mathrm r(T(t_0)), 1)$ such that $\norm{T(t_0)^n\rvert_{X_2}} < q^n$ for $n$ sufficiently large. As $\calM_P$ coincides with the spectral projection of $\calM_B = T(t_0)$ associated with the peripheral spectrum, it commutes with the operators of the semigroup $(T(t))_{t \geq 0}$ and hence the spaces $X_1$ and $X_2$ are both invariant under the action of $(T(t))_{t \geq 0}$. So $(T(t))_{t \geq 0}$ restricts to semigroups on $X_1$ and $X_2$, respectively. Moreover, $P$ is strictly positive on $E$ by Lemma~\ref{lemma:value-space-splitting}(iii) and hence so is $\calM_P$. In particular, $X_1$ is a Banach lattice by \cite[Proposition~III.11.5]{Schaefer1974} and the semigroup $(T(t)\rvert_{X_1})_{t \geq 0}$ is clearly positive in this situation.

Finally, $(T(t)\rvert_{X_1})_{t \geq 0}$ extends to a group on $X_1$. $\norm{T(t_0)^n\rvert_{X_2}} < q^n$ for $n$ big enough and $(T(t)\rvert_{X_2})_{t \geq 0}$ is bounded, it follows that $\norm{T(t)\rvert_{X_2}} \leq M \ue^{-\omega t}$ for suitable constants $M \geq 1$ and $\omega > 0$. Therefore, $(T(t))_{t \geq 0}$ is asymptotically exponentially periodic in the sense of Definition~\ref{def:asymptotically-periodic}.
\end{proof}

As the proof of the theorem above only used results from Section~\ref{section:preliminaries} and some spectral theory on the Banach space in which the $L^p$-functions took their values in, one can obtain the Theorem~\ref{theorem:asymptotic-periodicity-mult-op-semigroup-lp} also on spaces of measures due to the observations made in Remark~\ref{remark:multiplication-operators-on-measures}. For later reference we state here this version of the theorem explicitly. 

\begin{theorem} \label{theorem:asymptotic-periodicity-mult-op-semigroup-measures}
Let $\Omega = (\Omega, \Sigma)$ be a measurable space, $E$ a Banach lattice and $(T(t))_{t \geq 0}$ a bounded, positive semigroup on $X := M(\Omega; E)$, $1 \leq p \leq \infty$. Suppose that there is $t_0 \in [0, \infty)$ such that $T(t_0) = \calM_B$ for some irreducible operator $B$ on $E$ with $r(B) = 1$. If $r(B)$ is a pole of the resolvent of $B$, then there is a strictly positive projection $P$ commuting with the semigroup $(T(t))_{t \geq 0}$ with the following properties:
\begin{enumerate}[\upshape (i)]
\item $(T(t)\rvert_{PX})_{t \geq 0}$ can be extended to a positive, periodic group on the Banach lattice $PX$.
\item $(T(t)\rvert_{\ker P})_{t \geq 0}$ is uniformly exponentially stable, i.e., there exist $M \geq 1$ and $\omega > 0$ such that $\norm{T(t) - T(t) P}  \leq M \ue^{-\omega t}$. 
\end{enumerate}
\end{theorem}

\section{Applications to transport equations in infinite networks}
\label{section:applications-to-transport-equations-in-infinite-networks}

\subsection{Transport equation on infinite networks}

In this section we consider transport processes on infinite networks. The network will be modeled by an infinite, directed graph $G = (V, E)$ which is assumed to be
\begin{enumerate}
\item[(a)] \textit{simple}, i.e., the graph contains no loops and no multiple edges,
\item[(b)] \textit{locally finite}, i.e., each vertex only has finitely many incident edges,
\item[(c)] \textit{non-degenerate}, i.e., each vertex of the network has at least one incoming as well as at least one outgoing edge.
\end{enumerate}
Moreover, $G$ will be considered as a \textit{metric graph} by identifying each edge with the unit interval $[0, 1]$ and parametrizing it contrarily to its direction. This means, under this identification, that each edge is assumed to have its endpoint at $0$ and its starting point at $1$. 

The distribution of mass transported along one edge $e_j$, $j \in J \subseteq \N$, at some time $t \geq 0$ will be described by a function $u_j(t, x)$ for $x \in [0, 1]$. The material in the network will be transported along $e_j$ with constant velocity $c_j > 0$ and we assume that
\begin{align*}
0 < c_{\min} \leq c_j \leq c_{\max} < \infty.
\end{align*}
Furthermore, we define 
\begin{align*}
\bbB^C := C^{-1} \bbB C,
\end{align*}
where $C := \diag(c_j)$ denotes the diagonal velocity matrix. In each vertex the material is distributed to the outgoing edges as governed by the weights of the respective edges. This will be modeled by assuming that the functions $u_j$ satisfy the generalized Kirchhoff law 
\begin{align*}
\sum_{j \in J} \phi^-_{ij} c_j u_j(1, t) = \sum_{j \in J} \phi^+_{ij} c_j u_j(0, t)
\end{align*}
for all $i \in I$ and $t > 0$. Altogether this transport process is modelled by the partial differential equation
\begin{align} \label{eq:transport-equation-network}
\left\lbrace 
\begin{aligned}
\frac{\partial}{\partial t} u_j(t, x) &= c_j \frac{\partial}{\partial x} u_j(t, x), &&\qquad x \in (0, 1), \, t \geq 0, \\
u_j(0, x) &= f_j(x), &&\qquad x \in (0, 1),\\
u_j(1, t) &= \sum_{k \in J} \bbB^C_{jk} u_k(0, t), &&\qquad t \geq 0,\\
\end{aligned}
\right.
\end{align}
$j \in J$, where $f_j$ are the initial distributions of mass along the edges. Solutions to this equation were investigated by using semigroup theory mainly in two different settings:

(i) It is shown in \cite{Dorn2008} that the solutions of \eqref{eq:transport-equation-network} form a strongly continuous semigroup on the space $L^1([0, 1]; \ell^1)$ in the case that $c_j = 1$ for all $j \in J$.

(ii) In \cite{Budde2019} it is shown that the solutions of \eqref{eq:transport-equation-network} form a so-called bi-continuous semigroup on the space $L^\infty([0, 1]; \ell^1)$ under some conditions on the velocities $c_j$. 

Although the same equation is investigated (however in different spaces),
in \cite{Dorn2008} the asymptotic behaviour of the solutions is discussed, whereas no such results can be found in \cite{Budde2019} for the solutions in the bi-continuous case. By making use of the theory developed in the prior section we close this gap. 

But first, we revisit Dorn's asymptotics result \cite[Proposition~4.3]{Dorn2008}. Our argument shows that the asymptotical periodicity of the solution obtained by Dorn does not depend on the strong continuity of the semigroup. On the other hand, the exponential speed of convergence in the theorem does rely on the strong continuity. In addition, it is proven that the involved projection is strictly positive.

\subsection{Strongly continuous flows on $L^1([0,1]; \ell^1)$} \label{subsection:flows-on-l1}

Let us first assume that 
\begin{align*}
c_j = 1 \qquad \text{for all } j \in J
\end{align*}
and equip the space $L^1([0,1]; \ell^1)$ with the usual norm given by 
\begin{align*}
\norm{f} := \int_0^1 \norm{f(s)}_{\ell^1} \, \ud s \qquad (f \in L^1([0,1]; \ell^1)). 
\end{align*}
Consider the differential operator
\begin{align*}
A := \diag\bigg(\frac{\ud}{\ud x} \bigg), \quad \dom A := \set{f \in W^{1, 1}([0, 1]; \ell^1) : f(1) = \bbB f(0)}.
\end{align*}
Notice that $W^{1, 1}([0, 1]; \ell^1)$ embeds continuously into $C([0, 1]; \ell^1)$ so that the above operator is well-defined. By \cite[Proposition~3.1]{Dorn2008}, we can translate \eqref{eq:transport-equation-network} into the abstract Cauchy problem 
\begin{align}
\left\lbrace 
\begin{aligned}
u'(t) &= A u(t), &&\quad t \geq 0, \\
u(0) &= (f_j)_{j \in J} &&
\end{aligned}
\right.
\end{align}
on the space $L^1([0,1]; \ell^1)$. Now \cite[Proposition~3.3]{Dorn2008} shows that $A$ generates a strongly continuous semigroup $(T(t))_{t \geq 0}$ on $L^1([0,1]; \ell^1)$ given by
\begin{align*}
T(t) f(s) = \bbB^n f(s + t - n) \qquad (f \in L^1([0,1]; \ell^1)),
\end{align*}
if $n \in \N_0$ is such that $n \leq s + t < n + 1$. Moreover, it is easy to see that this semigroup is both contractive and positive. So $(T(t))_{t \geq 0}$ is basicly a shift semigroup on the graph with a ``jump" at the end of each edge governed by the weighted adjencency matrix $\bbB$.

\begin{proposition} \label{proposition:l1-asymptotics-simple-case}
Let $X:= L^1([0, 1]; \ell^1)$. Suppose that $\bbB$ is quasi-compact and irreducible. Then there is a strictly positive projection $P$ commuting with the semigroup $(T(t))_{t \geq 0}$ with the following properties:
\begin{enumerate}[\upshape (i)]
\item $(T(t)\rvert_{PX})_{t \geq 0}$ can be extended to a positive, periodic $C_0$-group on the Banach lattice $PX$.
\item $(T(t)\rvert_{\ker P})_{t \geq 0}$ is uniformly exponentially stable, i.e., $\norm{T(t) - T(t) P} \leq M \ue^{-\omega t}$ for some constants $M \geq 1$ and $\omega > 0$. 
\end{enumerate}
\end{proposition}
\begin{proof}
Clearly, $(T(t))_{t \geq 0}$ is a positive, contractive semigroup on $L^1([0,1]; \ell^1)$ with $T(1) = \calM_\bbB$. Moreover, $\bbB$ is irreducible by hypothesis with $r(\bbB) = 1$. Since $\bbB$ is assumed to be quasi-compact, the peripheral spectrum $\sigma(\bbB) \cap \bbT$ consists of poles of the resolvent. Hence, by Theorem~\ref{theorem:asymptotic-periodicity-mult-op-semigroup-lp} there is a strictly positive projection $P$ on $X$ such that $(T(t)\rvert_{PX})_{t \geq 0}$ can be extended to a positive, periodic $C_0$-group on the Banach lattice $PX$ and $(T(t)\rvert_{\ker P})_{t \geq 0}$ is uniformly exponentially stable, i.e., there are constants $M \geq 1$ and $\omega > 0$ such that $\norm{T(t) - T(t) P} \leq M \ue^{-\omega t}$.
\end{proof}

\begin{remark}
The strict positivity of the projection $P$ in the theorem above could alternatively be deduced from the irreducibility of the semigroup $(T(t))_{t \geq 0}$ (cf.\ \cite[Proposition~4.9]{Dorn2008}). 

However, on other state spaces where the solution semigroup is not strongly continuous (see Section~\ref{subsection:Bi-continuous-flows-on-L-infty}) there is no characterization of irreduciblity of the semigroup via the resolvent of the generator to the best of the authors knowledge.
\end{remark}

Now let us consider the case where not all velocities $c_j$ are equal to $1$. However, assume that the velocities $c_j$ are linearly dependant over $\Q$, i.e., $\frac{c_i}{c_j} \in \Q$ for all $i, j \in J$, with a common finite multiplier $c > 0$. This means that 
\begin{align*}
\ell_j := \frac{c}{c_j} \in \N \qquad \text{for all } j \in J.
\end{align*}
Consider the differential operator 
\begin{align*}
A_C :=  \diag\bigg(c_j \cdot \frac{\ud}{\ud x} \bigg), \quad \dom A_C := \set{f \in W^{1, 1}([0, 1]; \ell^1) : f(1) = \bbB^C f(0)}.
\end{align*}
The assumption on the velocities allows us to construct a new graph by adding $\ell_j - 1$ vertices on each edge $e_j$, $j \in J$ of the graph and let the so-created new edges inherit the orientation of the original edge. Carrying out this construction one can show  that the operator $A_C$ generates a strongly continuous semigroup $(T_C(t))_{t \geq 0}$ on a space isometrically isomorphic to $L^1([0, 1]; \ell^1)$, i.e., there exists a strictly positive isomorphism $\bbS \in \calL(L^1([0, 1]; \ell^1))$ such that
\begin{align*}
T_C(ct) f = \bbS^{-1} T(t) \bbS f \qquad (f \in L^1([0, 1]; \ell^1))
\end{align*}
for all $t \geq 0$. Here, $(T(t))_{t \geq 0}$ denotes the semigroup defined above. For more details on the construction we refer to \cite[Section~3]{Banasiak2014}; or see \cite{Namayanja2013} for even more detail.

Now, the following asymptotics result is an immediate consequence of Proposition~\ref{proposition:l1-asymptotics-simple-case}. 
\begin{corollary}
Let $X:= L^1([0, 1]; \ell^1)$. Suppose that $\bbB$ is quasi-compact and irreducible. Then there is a strictly positive projection $P_C$ commuting with the semigroup $(T_C(t))_{t \geq 0}$ with the following properties:
\begin{enumerate}[\upshape (i)]
\item $(T_C(t)\rvert_{P_CX})_{t \geq 0}$ can be extended to a positive, periodic $C_0$-group on the Banach lattice $P_C X$.
\item $(T(t)\rvert_{\ker P_C})_{t \geq 0}$ is uniformly exponentially stable. 
\end{enumerate}
\end{corollary}

\begin{remark}
By Proposition~\ref{proposition:graph-properties-vs-bbB-properties} the above results are applicable to each strongly connected graph with an attractor. In particular, the flows on such graphs are asymptotically periodic.
\end{remark}

\subsection{Bi-continuous flows on $L^\infty([0,1]; \ell^1)$} \label{subsection:Bi-continuous-flows-on-L-infty} Now we turn our attention to the bi-continuous case considered in \cite{Budde2019}. The abstract theory developed in Section~\ref{section:abstract-asymptotics-result} allows to prove analogous results to those in the $L^1$-setting. Again, we assume first that 
\begin{align*}
c_j = 1 \qquad \text{for all } j \in J
\end{align*}
and consider the space $L^\infty([0,1]; \ell^1)$ equipped with the usual norm
\begin{align*}
\norm{f} := \operatorname{esssup}_{s \in [0, 1]} \norm{f(s)}_{\ell^1} \qquad (f \in L^\infty([0,1]; \ell^1)). 
\end{align*}
Moreover, we consider the operator 
\begin{align*}
A := \diag\bigg(\frac{\ud}{\ud x} \bigg), \quad \dom A := \set{f \in W^{1, \infty}([0, 1]; \ell^1) : f(1) = \bbB f(0)},
\end{align*}
notice that it is well-defined since $W^{1, \infty}([0, 1]; \ell^1)$ embeds continuously into $C([0, 1]; \ell^1)$ and that \eqref{eq:transport-equation-network} can be translated into the abstract Cauchy problem 
\begin{align}
\left\lbrace 
\begin{aligned}
u'(t) &= A u(t), &&\quad t \geq 0, \\
u(0) &= (f_j)_{j \in J}. &&
\end{aligned}
\right.,
\end{align}
but this time on the space $L^\infty([0,1]; \ell^1)$. Completely analogously to \cite[Theorem~3.2]{Dorn2008} one obtains the following resolvent formula on $L^\infty([0,1]; \ell^1)$ (see also \cite[Proposition~3.8]{Budde2019}).

\begin{theorem} \label{theorem:resolvent-formula-transport-semigroup}
For $\lambda > 0$ one has 
\begin{align*}
(R(\lambda, A) f)(s) = \sum_{n = 0}^\infty \ue^{-\lambda n} \int_0^1 \ue^{-\lambda(t + 1 - s)} \bbB^{n + 1} f(t) \, \ud t + \int_s^1 \ue^{\lambda (s - t)} f(t) \, \ud t
\end{align*}
for all $f \in L^\infty([0,1]; \ell^1)$ and $s \in [0, 1]$. Furthermore, for $\lambda \in \C$ one has the characteristic equations
\begin{align}
\lambda \in \sigma(A) \quad \, \Longrightarrow& \quad \ue^\lambda \in \sigma(\bbB) \label{eq:char-eq-spectrum}, \\
\lambda \in \sigma_p(A) \quad \Longleftrightarrow& \quad \ue^\lambda \in \sigma_p(\bbB) \label{eq:char-eq-point-spectrum}.
\end{align}
\end{theorem}

As seen in the section before, the solution semigroup of \eqref{eq:transport-equation-network} on $L^1([0,1]; \ell^1)$ is basicly a shift with some jump. So would be a  reasonable guess that something similar if not the same stays true for $L^\infty([0,1]; \ell^1)$. However, it is well known that the shift semigroups are generally not strongly continuous on $L^\infty$-spaces. Hence, one can not expect that $A$ generates a strongly continuous semigroup. However, it is proven in \cite[Theorem~3.9]{Budde2019} that $A$ generates a semigroup with a weaker regularity property, a so called bi-continuous semigroup (see e.g.\ \cite{Kuehnemund2003} for a definition).

\begin{theorem}
The operator $A$ generates a contractive, bi-continuous semigroup $(T(t))_{t \geq 0}$ on $L^\infty([0,1]; \ell^1)$ given by
\begin{align*}
T(t) f(s) = \bbB^n f(s + t - n) \qquad (f \in L^\infty([0,1]; \ell^1)),
\end{align*}
if $n \in \N_0$ is such that $n \leq s + t < n + 1$, with respect to the weak$^\ast$-topology. 
\end{theorem}

In \cite[Theorem~4.10]{Dorn2008} Dorn gives a characterization for asymptotic periodicity of the  strongly continuous flow semigroup on $L^1([0,1]; \ell^1)$. Using the abstract theory from Section~\ref{section:abstract-asymptotics-result} and ideas from the proof of \cite[Theorem~4.10]{Dorn2008}, we can prove that the same characterization holds for the bi-continuous semigroup $(T(t))_{t \geq 0}$.

\begin{theorem} \label{theorem:asymptotic-behaviour-flow-semigroup}
Let $G$ be a strongly connected graph. Then the following two assertions are equivalent:
\begin{enumerate}[\upshape (i)]
\item $G$ has an attractor.
\item $(T(t))_{t \geq 0}$ is asymptotically exponentially periodic in the sense of Definition~\ref{def:asymptotically-periodic}.
\end{enumerate} 
In this case, the period $\theta$ of $(T(t)\rvert_{X_1})_{t \geq 0}$ is given by the greatest common divisor of all cycle lengths which occur in the graph $G$.  
\end{theorem}

\begin{proof}
(i) $\Rightarrow$ (ii): As $G$ is strongly connected and has an attractor, $\bbB$ is irreducible and quasi-compact. In particular, the peripheral spectrum of $\bbB$ consists of poles of the resolvent of $\bbB$. So it follows from Theorem~\ref{theorem:asymptotic-periodicity-mult-op-semigroup-lp} that $(T(t))_{t \geq 0}$ is asymptotically periodic since the flow semigroup is bounded and positive with $T(1) = \calM_\bbB$. Moreover, it follows from the proof of Theorem~\ref{theorem:asymptotic-periodicity-mult-op-semigroup-lp} that $\theta$ is smaller or equal to the index of imprimitivity of the matrix $k$, i.e., by the number of spectral values of $\bbB$ with modulus equal to $1$. But due to the characteristic equation~\eqref{eq:char-eq-point-spectrum} and the spectral mapping theorem for the point spectrum (see \cite[Corollary~1.5.2]{Farkas2003}) it follows that $\theta = k$. Hence, by the arguments given in the first part of the proof of \cite[Theorem~4.10]{Dorn2008}, $\theta$ equals the greatest common divisor of all cycle lengths which occur in the graph $G$.

(ii) $\Rightarrow$ (i): Suppose that there exist a strictly positive projection $P$ on $X$ such that $(T(t))_{t \geq 0}$ commutes with $P$ and satisfies the assertions from Definition~\ref{def:asymptotically-periodic}. Then $X = PX \oplus \ker P$ and we have 
\begin{align*}
\sigma(T(1)) = \sigma(T(1)|_{PX}) \cup \sigma(T(1)|_{\ker P}).
\end{align*}
Moreover, $\sigma(T(1)|_{PX}) \subseteq \bbT$ as $T(1)|_{PX}$ is doubly power-bounded and $\sigma(T(1)|_{\ker P})$ is contained in a ball with center $0$ with radius $0 < r < 1$ because $(T(1))^n|_{\ker P}$ tends to $0$ with respect to the operator norm as $n \to \infty$. Moreover, we have $\sigma(\bbB) = \sigma(T(1))$ and therefore, the peripheral spectrum of $\bbB$ is isolated. So denote by $\bbP$ the spectral projection onto the peripheral spectrum of $\bbB$. Then $\calM_\bbP$ is the spectral projection onto the peripheral spectrum of $T(1)$ and it follows from the uniqueness of the spectral projection that $P = \calM_\bbP$. 

Moreover, $(T(t)\rvert_{PX})_{t \geq 0}$ is periodic. Denote the period of $(T(t)\rvert_{PX})_{t \geq 0}$  in the following by $\theta$. Since $(T(t)\rvert_{PX})_{t \geq 0}$ is periodic, one has $1 \in \sigma_p(\bbB \rvert_{\bbT})$. Hence, \cite[Proposition~2.1]{Keicher2008a} shows that there exists a strictly positive fixed vector $x \in \ell^1$ of $\bbB$ and this fixed vector is clearly an element of $\bbP \ell^1$. In particular, $f \otimes x$ is an element of $L^\infty([0,1]; \bbP \ell^1)$ for each $f \in L^\infty([0,1])$. Now choose $k \in \N$ such that $k \leq \theta < k + 1$. Suppose that $k < \theta$ and fix $\varepsilon < \theta - k$. Then
\begin{align*}
\one_{[0, \varepsilon]}(s) x &= (\one_{[0, \varepsilon]} \otimes \, x) (s) = T(\theta) (\one_{[0, \varepsilon]} \otimes \, x) (s) \\
&= \one_{[0, \varepsilon]}(\theta + s - k) \bbB^k x = \one_{[0, \varepsilon]}(\theta + s - k) x
\end{align*}
for a.e.\ $s \in [0,1]$ which is absurd. Hence, we infer that $\theta = k$ which means that $T(1)\rvert_{PX}$ is algebraic, i.e., mapped to $0$ by a polynomial. Therefore, the spectrum of $T(1)\rvert_{PX}$ is finite and consists of poles of the resolvent and, by Proposition~\ref{prop:mult-op-spectral-properties}(iii), the same holds true for the spectrum of $\bbB\rvert_{\bbP \ell^1}$. Now it follows just as in the proof of \cite[Theorem~4.10]{Dorn2008} that the peripheral spectrum of $\bbB$ consists of entirely of first order poles which implies that $\bbB$ is quasi-compact.
\end{proof}

Consider again the case where not all velocities $c_j$ are equal to $1$ but the velocities $c_j$ are linearly dependant over $\Q$, i.e., $\frac{c_i}{c_j} \in \Q$ for all $i, j \in J$, with a common finite multiplier $c > 0$, i.e.,
\begin{align*}
\ell_j := \frac{c}{c_j} \in \N \qquad \text{for all } j \in J.
\end{align*}
Analogously to the $L^1$-case, consider the differential operator 
\begin{align*}
A_C :=  \diag\bigg(c_j \cdot \frac{\ud}{\ud x} \bigg), \quad \dom A_C := \set{f \in W^{1, \infty}([0, 1]; \ell^1) : f(1) = \bbB^C f(0)}.
\end{align*}
By using the same construction as mentioned in the former section one shows that the operator $A_C$ generates a bi-continuous semigroup $(T_C(t))_{t \geq 0}$ on $L^\infty([0, 1]; \ell^1)$ (see \cite[Proposition 3.11]{Budde2019}) and that there exists a strictly positive isomorphism $\bbS \in \calL(L^\infty([0, 1]; \ell^1))$ such that
\begin{align*}
T_C(ct) f = \bbS^{-1} T(t) \bbS f \qquad (f \in L^1([0, 1]; \ell^1))
\end{align*}
for all $t \geq 0$. Consequently, we obtain the following result. 

\begin{corollary}
Let $X:= L^\infty([0, 1]; \ell^1)$. Suppose that $\bbB$ is quasi-compact and irreducible. Then there is a strictly positive projection $P_C$ commuting with the semigroup $(T_C(t))_{t \geq 0}$ with the following properties:
\begin{enumerate}[\upshape (i)]
\item $(T_C(t)\rvert_{P_CX})_{t \geq 0}$ can be extended to a positive, periodic $C_0$-group on the Banach lattice $P_C X$.
\item $(T(t)\rvert_{\ker P_C})_{t \geq 0}$ is uniformly exponentially stable. 
\end{enumerate}
\end{corollary}

\section{Extension to $M([0,1]; \ell^1)$} \label{section:applications-to-transport-equations-in-infinite-networks-measures}

In this section, it is shown that the semigroup from Section~\ref{subsection:flows-on-l1} can be naturally extended to the space of $\ell^1$-valued measures of bounded variation and it is be shown that this extension has still some regularity properties. Furthermore, the abstract argument from Section~\ref{section:abstract-asymptotics-result} can be used to show that this extended semigroup is still asymptotically periodic. This is motivated by \cite{Matrai2007}, where something similar was done for a different kind of flow on finite networks. Furthermore, the asymptotic behaviour of such an extended semigroup was investigated in \cite[Corollary~4.17]{Matrai2007} (just for the case of finite networks) but this investigation relied on a non-constructive argument. 

For the introduction of $\ell^1$-valued measures of bounded variation and for the used notation, we refer the reader to Appendix~\ref{appendix:vector-valued-measures-and-the-bartle-integral}.

Define the operators $S(t): M([0, 1]; \ell^1) \to M([0, 1]; \ell^1)$, $t \geq 0$, by
\begin{align*}
(S(t) \mu) = \delta_{-(t - n)} \ast (\calM_{\bbB^n} \mu\rvert_{[t - n, 1]}) + \delta_{n + 1 - t} \ast (\calM_{\bbB^{n + 1}}\mu\rvert_{[0, t - n)}),
\end{align*}
where $n \in \N_0$ such that $n \leq t < n + 1$. In particular, one has
\begin{align*}
(S(t) \mu)(A) &= \bbB^n \mu((A \cap [t - n, 1]) - (t - n)) \\
& \qquad + \bbB^{n + 1} \mu((A \cap [0, t - n)) + n + 1 - t),
\end{align*}
for each Borel measurable set $A \in \calB([0, 1])$. It is not hard to see that $(S(t))_{t \geq 0}$ is a contractive semigroup on $M([0, 1]; \ell^1)$.

\begin{lemma}
$(S(t))_{t \geq 0}$ defines a contractive semigroup on $M([0, 1]; \ell^1)$ such that $S(1) = \calM_\bbB$. 
\end{lemma}

\begin{proof}
It is easy to see that $(S(t))_{t \geq 0}$ defines a semigroup on $M([0, 1]; \ell^1)$ such that $S(1) = \calM_\bbB$. 

To see that $(S(t))_{t \geq 0}$ is contractive, let $\mu \in  M([0, 1]; \ell^1)$ and $A_1, \dots, A_N \subseteq [0, 1]$ measurable, pairwise disjoint sets such that $[0, 1] = \bigcup_{k = 1}^N A_k$. Now let $t > 0$ and pick $n \in \N$ such that $n \leq t < n + 1$ and consider the sets
\begin{align*}
B_k := (A_k \cap [t - n, 1]) - (t - n), \quad C_k := (A_k \cap [0, t - n)) + (1 + n - t)
\end{align*}
for $k = 1, \dots, N$. Observe that all these sets are pairwise disjoint and $[0, 1] = \bigcup_{n \in \N} B_k \cup C_k$. Hence,
\begin{align*}
\sum_{k = 1}^N \norm{S(t) \mu(A_k)}_1 &\leq \sum_{k = 1}^N \norm{\bbB^n \mu(B_k)}_1 + \sum_{k = 1}^N \norm{\bbB^{n + 1} \mu(C_k)}_1 \\
&\leq \sum_{k = 1}^N \norm{\mu(B_k)}_1 + \sum_{k = 1}^N \norm{\mu(C_k)}_1 \leq \norm{\mu}
\end{align*}
and therefore $\norm{S(t) \mu} \leq \norm{\mu}$ by taking the supremum over all finite partitions of the unit interval $[0,1]$.
\end{proof}

In the following, we call $(S(t))_{t \geq 0}$ the \textit{extended flow semigroup} on $M([0,1]; \ell^1)$. This is motivated by the following result which shows that the semigroup $(S(t))_{t \geq 0}$ actually extends the semigroup $(T(t))_{t \geq 0}$ on $L^1([0, 1]; \ell^1)$ from Section~\ref{subsection:flows-on-l1} for the case that all velocities $c_j$ are equal to $1$. To clarify what we mean by this extension, note if $f \in L^1([0, 1]; \ell^1)$, then 
\begin{align*}
(f \,\ud s): \calB([0, 1]) \to \ell^1, \quad (f \, \ud s)(A) = \int_A f \, \ud s
\end{align*}
defines a $\ell^1$-valued measure of bounded variation with $\norm{f \, \ud s} = \norm{f}_1$. Hence,
\begin{align*}
L^1([0, 1]; \ell^1) \to M([0, 1]; \ell^1), \quad f \mapsto f \, \ud s
\end{align*}
is a linear isometry and therefore, in what follows, elements of $L^1([0, 1]; \ell^1)$ will be freely identified with their respective images in $M([0, 1]; \ell^1)$ under this embedding.

\begin{proposition}
The identity
\begin{align*}
(T(t) f) \, \ud s = S(t) (f \, \ud s)
\end{align*}
holds for all $f \in L^1([0, 1]; \ell^1)$.
\end{proposition}

\begin{proof}
Let $f \in L^1([0, 1]; \ell^1)$. By Dynkin's Lemma~\ref{lemma:dynkin} it is enough to show that both measures $(T(t) f) \, \ud s$ and $S(t) (f \, \ud s)$ coincide on all intervals $[a, b]$, where $0 \leq a \leq b \leq 1$ and by the semigroup law it suffices to suppose $0 \leq t < 1$. Firstly, one has
\begin{align*}
(T(t) f) \, \ud s([a, b]) &= \int_0^1 \one_{[a, b]}(s) T(t) f(s) \, \ud s \\
&= \int_{[a, 1 - t]} f(t + s) \, \ud s + \int_{[1 - t, b]} \bbB f(t + s - 1) \, \ud s.
\end{align*}
On the other hand,
\begin{align*}
S(t) (f \, \ud s)([a, b]) &= \int_0^1 \one_{[a, b]}(s) \, \ud (\delta_{-t} \ast (f \, ds)\rvert_{[t, 1]})(s) \\
&\quad + \int_0^1 \one_{[a, b]}(s) \, \ud (\delta_{1 - t} \ast \calM_\bbB (f \, ds)\rvert_{[t, 1]})(s) \\
&= \int_{[t, 1]} \one_{[a, b]}(s - t) f(s) \, \ud s + \int_{[0, t)} \one_{[a, b]}(s + 1 - t) \bbB f(s) \, \ud s \\
&= \int_{[0, 1 - t]} \one_{[a, b]}(s) f(s + t) \, \ud s + \int_{[1 - t, 1)} \one_{[a, b]}(s) \bbB f(s + t - 1) \, \ud s \\
&= \int_{[a, 1 - t]} f(s + t) \, \ud s + \int_{[1 - t, b]} \bbB f(s + t - 1) \, \ud s. \\
\end{align*}
Hence, $(T(t) f) \, \ud s = S(t) (f \, \ud s)$ for all $f \in L^1([0, 1]; \ell^1)$ and $t \geq 0$. 
\end{proof}
Now the regularity of $(S(t))_{t \geq 0}$ is investigated. In the appendix, the nilpotent left shift semigroup $(\tau_0(t))_{t \geq 0}$ on the space $M([0, 1]; E')$ is defined for an arbitrary Banach space $E$. For $E = c_0$, i.e., $E' = \ell^1$, we obtain the following regularity result as a direct consequence of Proposition~\ref{prop:nilpotent-shift-semigroup-continuity}.

\begin{lemma} \label{lemma:nilpotent-shift-semigroup-continuity-c-0}
The nilpotent left shift semigroup $(\tau_0(t))_{t \geq 0}$, given by 
\begin{align*}
(\tau_0(t) \mu)(A) := (\delta_{-t} \ast \mu\rvert_{[t, 1]})(A) = \mu((A \cap [t, 1]) - t) \qquad (A \in \calB([0, 1])),
\end{align*}
is weak$^\ast$-continuous on $M([0, 1]; \ell_1)$.
\end{lemma}

Now we turn back our attention to the flow semigroup and its regularity.

\begin{theorem}
The flow semigroup $(S(t))_{t \geq 0}$ is a weak$^\ast$-continuous semigroup on $M([0, 1]; \ell^1)$.
\end{theorem}

\begin{proof}
Let $f \in C([0, 1], c_0)$ and $\mu \in M([0, 1]; \ell^1)$. As in the proof of Proposition~\ref{prop:nilpotent-shift-semigroup-continuity}, we consider the trivial extension $\widetilde f \in B(\R; c_0)$ of $f$ to $\R$. Then for $t \in [0, 1)$, one has 
\begin{align*}
\abs{\dual{f, S(t) \mu - \mu}} &= \abs{\dual{f, \delta_{-t} \ast \mu\rvert_{[t, 1]}} + \dual{f, \delta_{1 - t} \ast \calM_\bbB \mu\rvert_{[0, t)}} - \dual{f, \mu}} \\
&\leq \abs{\dual{f, \tau_0(t) \mu  - \mu}} + \abs[\bigg]{\int_\R \int_\R \widetilde f(s + r) \, \ud (\calM_\bbB \mu\rvert_{[0, t)})(s) \, \ud \delta_{1 - t}(r)} \\
&= \abs{\dual{f, \tau_0(t) \mu - \mu}} + \abs[\bigg]{\int_\R \widetilde f(s + 1 - t) \, \ud (\calM_\bbB \mu\rvert_{[0, t)})(s)} \\
&= \abs{\dual{f, \tau_0(t) \mu - \mu}} + \abs[\bigg]{\int_{[0, t)} \widetilde f(s + 1 - t) \, \ud (\calM_\bbB \mu)(s)} \longrightarrow 0,
\end{align*}
where the first term turns goes $0$ due to Lemma~\ref{lemma:nilpotent-shift-semigroup-continuity-c-0} and the second by the dominated convergence theorem (Theorem~\ref{theorem:dominated-convergence}).
\end{proof}

Using the framework of operator-valued multiplication operators, one can obtain an asymptotic result for the extended flow semigroup similar to \cite[Proposition~4.3]{Dorn2008}. The proof works analogously to our proof of Proposition~\ref{proposition:l1-asymptotics-simple-case} with the only difference being that one needs to employ Theorem~\ref{theorem:asymptotic-periodicity-mult-op-semigroup-measures} instead of Theorem~\ref{theorem:asymptotic-periodicity-mult-op-semigroup-lp}. 

\begin{proposition} \label{proposition:m-asymptotics-simple-case}
Let $X:= M([0, 1]; \ell^1)$. Suppose that $\bbB$ is quasi-compact and irreducible. Then there is a strictly positive projection $P$ commuting with the semigroup $(S(t))_{t \geq 0}$ with the following properties:
\begin{enumerate}[\upshape (i)]
\item $(S(t)\rvert_{PX})_{t \geq 0}$ can be extended to a positive, periodic $C_0$-group on the Banach lattice $PX$.
\item $(S(t)\rvert_{\ker P})_{t \geq 0}$ is uniformly exponentially stable, i.e., $\norm{S(t) - S(t) P} \leq M \ue^{-\omega t}$ for some constants $M \geq 1$ and $\omega > 0$. 
\end{enumerate}
\end{proposition}

\appendix

\section{Vector valued measures and the Bartle integral} \label{appendix:vector-valued-measures-and-the-bartle-integral}

The Pettis integral is a well-known generalisation of the Bochner integral which has various applications in operator theory. The Bartle integral, although not as general as the Pettis integral, is another such generalization far less known. It has many applications to integration theory in vector-valued function spaces and will be used it frequently throughout Section~\ref{section:applications-to-transport-equations-in-infinite-networks-measures} of this paper. In this paper it is used to identify the duality of the space of vector valued continuous functions that vanish at infinity. This brief appendix shall serve as a short introduction to the topic.

\subsection{Vector valued measures} \label{subsection:vector-valued-measures}

Let $(\Omega, \Sigma)$ be a measurable space and let $E$ be a Banach space. A $\sigma$-additive function $\mu: \Sigma \to E$ is called \textit{vector measure} if $\mu(\emptyset) = 0$. If $\mu$ is a vector measure, its \textit{variation} $\abs{\mu}: \Sigma \to [0, \infty]$ is 
\begin{align*}
\abs{\mu}(A) = \sup_{\pi} \sum_{B \in \pi} \norm{\mu(B)},
\end{align*}
where the supremum is taken over all finite partitions $\pi$ of $A$. The total variation $\abs{\mu}$ is a positive measure with $\norm{\mu(A)} \leq \abs{\mu}(A)$ for all $A \in \Sigma$. Moreover, $\mu$ is said to be of \textit{bounded variation} if $\abs{\mu}(\Omega) < \infty$. The set of all vector measures $\mu: \Sigma \to E$ with bounded variation is denoted by $M(\Omega; E)$. $M(\Omega; E)$ is clearly a vector space and one can show that $\norm{\mu} := \abs{\mu}(\Omega)$ defines a norm on $M(\Omega; E)$ that renders $M(\Omega; E)$ a Banach space. Obviously, Dynkin's lemma \cite[Lemma~A.1.3]{Hytoenen2016} still holds for vector valued measures.

\begin{lemma}[Dynkin] \label{lemma:dynkin}
Let $\mu, \nu: \Sigma \to E$ be two vector measures and $\calD \subseteq \Sigma$ with the following properties:
\begin{enumerate}
\item[\upshape(a)] $\Omega \in \calD$.
\item[\upshape(b)] $A, B \in \calD$ implies $A \cap B \in \calD$.
\item[\upshape(c)] The $\sigma$-algebra $\sigma(\calD)$ generated by $\calD$ coincides with $\Sigma$. 
\end{enumerate}
If $\mu(A) = \nu(A)$ for all $A \in \calD$, then $\mu = \nu$. 
\end{lemma}

\subsection{The Bartle integral}
Let $(\Omega, \Sigma)$ be a measurable space and let $E$ be a Banach space. Then function $f : \Omega \to E$ is called \textit{simple} if there are finitely many vectors $x_1, \dots, x_n \in E$ and sets $A_1, \dots A_n \in \Sigma$ such that $f = \sum_{i = 1}^n \one_{A_i} x_i$. The space of simple functions on $\Omega$ with values in $E$ will be denoted by $S(\Omega; E)$. If $f \in S(\Omega; E)$ and $\mu \in M(\Omega; E')$ then 
\begin{align*}
\int_\Omega f \, \ud \mu := \sum_{i = 1}^n \dual{x_i, \mu(A_i)} \in \bbC
\end{align*}
is independent of the representation of $f$ and is called the \textit{Bartle integral} of $f$ with respect to $\mu$ (cf.\ \cite[Section~4]{Bartle1956}). In particular, the map
\begin{align*}
S(\Omega; E) \times M(\Omega; E') \to \bbC, \quad (f, \mu) \mapsto \int_\Omega f \, \ud \mu
\end{align*}
is bilinear and one has
\begin{align*}
\abs[\bigg]{\int_\Omega f \, \ud \mu} \leq \sum_{i = 1}^n \norm{x_i} \norm{\mu(A_i)} \leq \sum_{i = 1}^n \norm{x_i} \abs{\mu}(A_i) = \int_\Omega \norm{f} \, \ud \abs{\mu} \leq \norm{f}_\infty \norm{\mu}.
\end{align*}
Hence, the map admits a continuous bilinear extension of norm of most one to the space $B(\Omega; E) \times M(\Omega; E')$. Here $B(\Omega; E)$ denotes the space of \textit{Bartle integrable functions}, which is the closure of $S(\Omega; E)$ in the space of bounded $E$-valued functions on $\Omega$ with respect to the supremum norm. The following dominated convergence theorem follows directly from the scalar one.

\begin{theorem}[Dominated Convergence Theorem] \label{theorem:dominated-convergence}
Let $E$ be a Banach space and $\mu \in M(\Omega; E')$. Further let $(f_n)_{n \in \N}$ be a bounded sequence in $B(\Omega; E)$ and $f \in B(\Omega; E)$ such that $f_n \to f$ pointwise. Then
\begin{align*}
\lim_{n \to \infty} \int_\Omega f_n \, \ud \mu = \int_\Omega f \, \ud \mu.
\end{align*}
\end{theorem}

\subsection{The vector-valued Riesz representation theorem} 

Let $\Omega$ be a $\sigma$-compact metrizable space, $\calB(\Omega)$ the Borel $\sigma$-algebra on $\Omega$ and $E$ some Banach space. Since $B(\Omega; E)$ is the closure of $S(\Omega; E)$ with respect to the supremum norm, one easily sees that $C_0(\Omega; E) \subseteq B(\Omega; E)$. This observation is interesting because it allows to use the concept of the Bartle integral to generalize the Riesz Representation Theorem to the space $C_0(\Omega; E)$ (cf.\ \cite{Brooks1974} or \cite[Appendix, Theorem~2.2.4]{Amann1995}). 

\begin{theorem} \label{theorem:riesz-representation-theorem}
The mapping
\begin{align*}
M(\Omega; E') \to C_0(\Omega; E)', \quad \mu \mapsto \bigg(f \mapsto \int_\Omega f \, \ud \mu \bigg)
\end{align*}
is a isometric isomorphism. In particular, $M(\Omega; E')$ is isometrically isomorphic to the dual of $C_0(\Omega; E)$.
\end{theorem}

\subsection{Convolution of measures and nilpotent shifts} 
Let $E$ be a Banach space. For measures $\mu \in M(\R)$ and $\nu \in M(\R; E')$ we define their convolutional product $\ast: M(\R) \times M(\R; E') \to M(\R; E')$ via
\begin{align} \label{eq:integration-against-convolution-of-measures}
\int_\R f \, \ud (\mu \ast \nu) = \int_\R \int_\R f(s + r) \, \ud \nu(s) \, \ud \mu(r)
\end{align}
for all $f \in S(\R; E)$. By density of $S(\R; E)$ in $B(\R; E)$ the identity above holds even for all $f \in B(\R; E)$.

Consider the \textit{left shift group} $(\tau_t)_{t \in \R}$ on the space $M(\R; E')$, given by 
\begin{align*}
\tau_t: M(\R; E') \to M(\R; E'), \quad  (\tau_t \mu)(A) := \mu(A - t).
\end{align*}
Since the left shift $(\tau_t)_{t \in \R}$ on $M(\R; E')$ is the dual semigroup of the right shift semigroup on $C_0(\R; E)$ by Theorem~\ref{theorem:riesz-representation-theorem}, which is well-known to be strongly continuous, $(\tau_t)_{t \in \R}$ is a weak$^\ast$-continuous semigroup on $M(\R; E')$. In particular, one has
\begin{align} \label{eq:integration-against-shifted-measure}
\int_\R f(s) \, \ud(\tau_t \mu)(s) = \int_\R f(s - t) \, \ud \mu(s)
\end{align}
for all $f \in C_0(\R; E)$.
As a straightforward application of \eqref{eq:integration-against-convolution-of-measures}~and~\eqref{eq:integration-against-shifted-measure} one obtains the following lemma which shows that we can express shifts by convolutions with Dirac measures.

\begin{lemma}
The identity $\delta_{-t} \ast \mu = \tau_t \mu$ holds for all $\mu \in M(\R; E')$ and $t \in \R$.
\end{lemma}


%

Let $B$ be a closed subset of $\R$ and $\mu \in M(\R; E')$. Recall from that the measure $\mu\rvert_B \in M(B; E')$, defined by $\mu\rvert_B(A) := \mu(A \cap B)$, is called the \textit{restriction of $\mu$ to $B$}. Moreover, $\mu\rvert_B$ can be seen as an element of $M(\R; E')$ via identification with its canonical zero extension to $\R$. 

Using this identification, we can define the nilpotent left shift semigroup $(\tau_0(t))_{t \geq 0}$ on $M([0, 1]; E')$ by
\begin{align*}
(\tau_0(t) \mu)(A) := (\delta_{-t} \ast \mu\rvert_{[t, 1]})(A) = \mu((A \cap [t, 1]) - t),
\end{align*}
for all $A \in \calB([0, 1])$. Note that $\operatorname{supp}(\delta_{-t} \ast \mu\rvert_{[t, 1]}) \subseteq [0,1]$ for all $t \geq 0$ -- so the semigroup is indeed well-defined.

\begin{proposition} \label{prop:nilpotent-shift-semigroup-continuity}
Let $E$ be a separable Banach space. Then the nilpotent left shift semigroup $(\tau_0(t))_{t \geq 0}$ is weak$^\ast$-continuous on $M([0, 1]; E')$.
\end{proposition}

\begin{proof}
Let $f \in C([0, 1]; E)$ and consider its trivial extension
\begin{align*}
\widetilde f: \R \to E, \quad f(s) :=
\begin{cases}
f(s), \quad &\text{if } s \in [0, 1], \\
0, \quad &\text{else}.
\end{cases}
\end{align*}
Then $\widetilde f \in B(\R; E)$. So for $t \in [0, 1)$ one has
\begin{align*}
\abs{\dual{f, \tau_0(t) \mu - \mu}} &= \abs{\dual{\widetilde f, \tau_0(t) \mu} - \dual{f, \mu}} \\
&= \abs[\bigg]{\int_\R \widetilde f \, \ud(\delta_{-t} \ast \mu\rvert_{[t, 1]}) - \int_{[0, 1]} f \, \ud \mu} \\
&= \abs[\bigg]{\int_\R \int_\R \widetilde f(s + r) \, \ud \mu\rvert_{[t, 1]}(s) \, \ud \delta_{-t}(r) - \int_{[0, 1]} f(s) \, \ud \mu(s)} \\
&= \abs[\bigg]{\int_\R \widetilde f(s - t) \, \ud \mu\rvert_{[t, 1]}(s) - \int_{[0, 1]} f(s) \, \ud \mu(s)} \\
&= \abs[\bigg]{\int_{[t, 1 + t]} f(s - t) \, \ud \mu(s) - \int_{[0, 1]}f(s) \, \ud \mu(s)} \longrightarrow 0
\end{align*}
as $t \searrow 0$ by the dominated convergence theorem (Theorem~\ref{theorem:dominated-convergence}).
\end{proof}

\bibliographystyle{plain}
\bibliography{literature}

\end{document}